\documentclass[a4paper,12pt]{article}
\usepackage{authblk}
\usepackage{amsmath}
\usepackage{amsfonts}
\usepackage{amssymb}
\usepackage{amsthm}
\usepackage{graphicx}
\DeclareMathOperator{\tr}{tr}
\DeclareMathOperator{\Cof}{Cof}
\newcommand{\norm}[1]{\left|\!\left|#1\right|\!\right|}
\newcommand{\R}{\mathbb{R}}
\newcommand{\C}{\mathbb{C}}
\newcommand{\N}{\mathbb{N}}
\newcommand{\inv}{^{-1}}
\DeclareMathOperator{\dist}{dist}
\DeclareMathOperator{\dev}{dev}
\DeclareMathOperator{\sym}{sym}
\renewcommand{\Re}{\mathop{\mathfrak{Re}} }

\newcommand{\abs}[1]{|#1|}

\DeclareMathOperator{\SL}{SL}
\DeclareMathOperator{\SO}{SO}

\newtheorem{theorem}{Theorem}
\newtheorem{formulation}[theorem]{Theorem}
\newtheorem{lemma}[theorem]{Lemma}
\newtheorem{consequence}[theorem]{Consequence}
\newtheorem{remark}[theorem]{Remark}
\newtheorem{conj}[theorem]{Conjecture}
\theoremstyle{remark}
\newtheorem{ex}[theorem]{\rm \textbf{Example}}
\begin{document}
\title{Sum of squared logarithms - An inequality relating positive definite matrices and their matrix logarithm}
\author[1,2]{Mircea B\^{\i}rsan}
\author[1,3]{Patrizio Neff}
\author[1]{Johannes Lankeit}
\affil[1]{Lehrstuhl f\"ur Nichtlineare Analysis und Modellierung, Fakult\"at f\"ur Mathematik, Universit\"at Duisburg-Essen, Germany}
\affil[2]{Department of Mathematics, University ``A.I. Cuza'' of Ia\c si, Romania}
\footnotetext[3]{To whom correspondence should be addressed. e-mail: {\tt patrizio.neff@uni-due.de}}
\date{\today}
\maketitle

\begin{abstract}
 \noindent Let $y_1,y_2,y_3,a_1,a_2,a_3\in(0,\infty)$ be such that $ y_1\,y_2\,y_3= a_1\,a_2\,a_3$ and
 \begin{align*}
      y_1  +y_2 +y_3 \geq  a_1  + a_2 +a_3 ,\quad
      y_1 y_2  +y_2 y_3 +y_1 y_3 \geq  a_1 a_2  +a_2 a_3 +a_1 a_3.
 \end{align*}
 Then 
 \[
  (\log y_1)^2 + (\log y_2)^2 + (\log y_3)^2 \geq (\log a_1)^2 + (\log a_2)^2 + (\log a_3)^2.
 \]
 This can also be stated in terms of real positive definite $3\times3$-matrices $P_1,P_2$: If their determinants are equal $\det P_1=\det P_2$, then
 \[
  \tr P_1 \geq \tr P_2\text{ and } \tr\Cof P_1 \geq \tr\Cof P_2 \implies  \norm{\log P_1}_F^2\geq \norm{\log P_2}_F^2,
 \]
 where $\log$ is the principal matrix logarithm and $\norm{P}_F^2=\sum_{i,j=1}^3 P_{ij}^2$ denotes the Frobenius matrix norm. Applications in matrix analysis  and nonlinear elasticity are indicated.
 {\bf{AMS-Classification:} 26D05, 26D07}
\end{abstract}

\section{Introduction}
Convexity is a powerful source for obtaining new inequalities, see e.g. \cite{Guan06,Roventa12}. In applications coming from nonlinear elasticity we are faced, however, with variants of the squared logarithm function, see the last section. The function $(\log(x))^2$ is neither convex nor concave. Nevertheless the sum of squared logarithm inequality holds.
We will proceed as follows: In the first section we will give several equivalent formulations of the inequality, for example in terms of the coefficients of the characteristic polynomial (Theorem \ref{charpol}), in terms of elementary symmetric polynomials (Theorem \ref{elementarpolynome}), in terms of means (Theorem \ref{mittel}) or in terms of the Frobenius matrix norm (Theorem \ref{frobeniusnorm}). A proof of the inequality  will be given in Section 2, and  some counterexamples for slightly changed variants of the inequality are discussed in Section 3. In the last section an application of the sum of squared logarithms inequality in matrix analysis and in the mathematical theory of nonlinear elasticity is indicated.

\section{Formulations of the problem}
All theorems in this section are equivalent.
\begin{formulation}\label{charpol}
For $n=2$ or $n=3$ let $P_1, P_2 \in \R^{n\times n}$ be positive definite real matrices.
Let the coefficients of the characteristic polynomials of $P_1$ and $P_2$ satisfy
\[
 \tr P_1\geq \tr P_2 \text{ and } \tr\Cof P_1 \geq \tr\Cof P_2,\text{ and } \det P_1=\det P_2.
\]
Then
\[
 \norm{\log{P_1}}_F^2\geq \norm{\log{P_2}}_F^2.
\]
\end{formulation}
For $n=3$ we will now give equivalent formulations of this statement. The case $n=2$ can be treated analogously. For its proof see Remark \ref{remark6}.
By orthogonal diagonalization of $P_1$ and $P_2$ the inequalities can be rewritten in terms of the eigenvalues $y_1,y_2,y_3$ and $a_1,a_2,a_3$ respectively.

\begin{formulation}\label{ohnequadrat}
Let the real numbers $a_1,a_2,a_3>0$ and $y_1,y_2,y_3>0$ be such that
\begin{equation}\label{2}
\begin{array}{c}
      y_1  +y_2 +y_3 \geq  a_1  + a_2 +a_3 ,\\
      y_1 y_2  +y_2 y_3 +y_1 y_3 \geq  a_1 a_2  +a_2 a_3 +a_1 a_3 ,\\
      y_1y_2y_3= a_1a_2a_3.
      \end{array}
\end{equation}
Then
\begin{equation}\label{2bis}
    (\log y_1)^2 + (\log y_2)^2 + (\log y_3)^2 \geq (\log a_1)^2 + (\log a_2)^2 + (\log a_3)^2.
\end{equation}
\end{formulation}

The elementary symmetric polynomials, see e.g.~\cite[p.178]{Steele2004}
\begin{align*}
 e_0(y_1,y_2,y_3)&=1\\
 e_1(y_1,y_2,y_3)&=y_1+y_2+y_3\\
 e_2(y_1,y_2,y_3)&=y_1y_2+y_1y_3+y_2y_3\\
 e_3(y_1,y_2,y_3)&=y_1y_2y_3
\end{align*}
are known to have the Schur-concavity property (i.e. $-e_k$ is Schur-convex) \cite{Guan06,Khan2012}, see \eqref{schur_convex}. It is possible to express the problem in terms of these elementary symmetric polynomials as follows:
\begin{formulation}
\label{elementarpolynome}
Let $a_1,a_2,a_3>0$ and $y_1,y_2,y_3>0$ satisfy
\[
 e_1(y_1,y_2,y_3)\geq e_1(a_1,a_2,a_3),\quad e_2(y_1,y_2,y_3)\geq e_2(a_1,a_2,a_3),\quad e_3(y_1.y_2,y_3)=e_3(a_1,a_2,a_3).
\]
Then
\[
 e_1((\log y_1)^2,(\log y_2)^2,(\log y_3)^2)\geq e_1((\log a_1)^2,(\log a_2)^2,(\log a_3)^2).
\]
\end{formulation}

Because $y_1y_2y_3= a_1a_2a_3>0$, we have
$$y_1 y_2  +y_2 y_3 +y_1 y_3 \geq  a_1 a_2  +a_2 a_3 +a_1 a_3 \quad \Leftrightarrow\quad \dfrac{1}{y_1} + \dfrac{1}{y_2} + \dfrac{1}{y_3} \geq  \dfrac{1}{a_1} + \dfrac{1}{a_2} + \dfrac{1}{a_3}.$$ Thus we obtain

\begin{formulation}
  \label{neff_log_formulation}
Let the real numbers $a_1,a_2,a_3>0$ and $y_1,y_2,y_3>0$ be such that
\begin{equation}\label{2neff}
\begin{array}{c}
      y_1  +y_2 +y_3 \geq  a_1  + a_2 +a_3 ,\\
      \dfrac{1}{y_1} + \dfrac{1}{y_2} + \dfrac{1}{y_3} \geq  \dfrac{1}{a_1} + \dfrac{1}{a_2} + \dfrac{1}{a_3} , \\
         y_1y_2y_3= a_1 a_2 a_3.
      \end{array}
\end{equation}
Then
\begin{equation}\label{2bis}
    (\log y_1)^2 + (\log y_2)^2 + (\log y_3)^2 \geq (\log a_1)^2 + (\log a_2)^2 + (\log a_3)^2.
\end{equation}
\end{formulation}

The conditions \eqref{2neff} are also simple expressions in terms of arithmetic, harmonic and geometric and quadratic mean
\begin{align*}
 A(y_1,y_2,y_3)=\frac{y_1+y_2+y_3}{3},\;\;H(y_1,y_2,y_3)=\frac{3}{\frac{1}{y_1}+\frac{1}{y_2}+\frac{1}{y_3}},\\
 G(y_1,y_2,y_3)=\sqrt[3]{y_1y_2y_3},\;\; Q(y_1,y_2,y_3)=\sqrt{\frac{1}{3} (y_1^2+y_2^2+y_3^2)}
\end{align*}
\begin{formulation}\label{mittel}
Let $a_1,a_2,a_3>0$ and $y_1,y_2,y_3>0$. Then
$A(y_1,y_2,y_3)\geq A(a_1,a_2,a_3)$, $H(a_1,a_2,a_3)\geq H(y_1,y_2,y_3)$ (``reverse!'') and $G(y_1,y_2,y_3)=G(a_1,a_2,a_3)$ imply
\[
    Q(\log y_1,\log y_2,\log y_3)\geq Q(\log a_1,\log a_2, \log a_3)
\]
\end{formulation}

We denote by
$$a_i=:d_i^2\,,\qquad y_i=:x_i^2.$$
and arrive at

\begin{formulation}\label{mitquadraten}
Let the real numbers $d_i$ and $x_i$ be such that $d_1,d_2,d_3>0$, $x_1,x_2,x_3>0$ and
\begin{equation}\label{1}
    \begin{array}{l}
      x_1^2 +x_2^2+x_3^2\geq  d_1^2 +d_2^2+d_3^2,\vspace{4pt}\\
      x_1^2x_2^2 +x_2^2x_3^2+x_1^2x_3^2\geq  d_1^2d_2^2 +d_2^2d_3^2+d_1^2d_3^2,\vspace{4pt}\\
      x_1x_2x_3\,=\, d_1d_2d_3.
    \end{array}
\end{equation}
Then
\begin{equation}\label{1bis}
    (\log x_1)^2 + (\log x_2)^2 + (\log x_3)^2 \geq (\log d_1)^2 + (\log d_2)^2 + (\log d_3)^2.
\end{equation}
\end{formulation}

If we again view $x_i$ and $d_i$ as eigenvalues of positive definite matrices, an equivalent formulation of the problem can be given in terms of their Frobenius matrix norms:

\begin{formulation}\label{frobeniusnorm}
For $n\in\{2,3\}$ let $P_1, P_2 \in \R^{n\times n}$ be positive definite real matrices.
Let
\[
 \norm{P_1}_F^2 \geq \norm{P_2}_F^2 \text{ and } \norm{P_1\inv}_F^2 \geq \norm{P_2\inv}_F^2, \text{ and } \det P_1=\det P_2.
\]
Then
\[
 \norm{\log{P_1}}_F^2\geq \norm{\log{P_2}}_F^2.
\]
\end{formulation}

Let us reconsider the formulation from Theorem \ref{mittel}. If we denote
$$c_i:=\log a_i\,,\qquad z_i:=\log y_i\,,$$
from $H(a_1,a_2,a_3)\geq H(y_1,y_2,y_3)$ we obtain
$$ e^{-z_1}   + e^{-z_2} + e^{-z_3} \geq  e^{-c_1}   + e^{-c_2} + e^{-c_3}.$$

\begin{formulation}\label{exp}\label{P2}
Let the real numbers $c_1, c_2, c_3 $ and $z_1, z_2, z_3 $ be such that
\begin{equation}
   \label{3}
    \begin{array}{l}
      e^{z_1}  +e^{z_2} +e^{z_3} \geq   e^{c_1}  +e^{c_2} +e^{c_3} ,\vspace{4pt}\\
      e^{-z_1}   + e^{-z_2} + e^{-z_3} \geq  e^{-c_1}   + e^{-c_2} + e^{-c_3},\vspace{4pt}\\
      z_1+z_2+z_3= c_1+c_2+c_3.
    \end{array}
\end{equation}
Then
\begin{equation}\label{3bis}
    z_1^2+z_2^2+z_3^2\geq c_1^2+c_2^2+c_3^2\,.
\end{equation}
\end{formulation}

In order to prove Theorem \ref{exp}, one can assume without loss of generality that
\begin{equation}\label{3,1}
    z_1+z_2+z_3= c_1+c_2+c_3=0.
\end{equation}
Thus, we have the equivalent formulation
\begin{formulation}\label{exp2}\label{P2}
Let the real numbers $\bar{c}_1, \bar{c}_2, \bar{c}_3 $ and $\bar{z}_1, \bar{z}_2, \bar{z}_3 $ be such that
\begin{equation}\label{3,2}
    \begin{array}{l}
      e^{\bar{z}_1}  +e^{\bar{z}_2} +e^{\bar{z}_3} \geq   e^{\bar{c}_1}  +e^{\bar{c}_2} +e^{\bar{c}_3} ,\vspace{4pt}\\
      e^{-\bar{z}_1}   + e^{-\bar{z}_2} + e^{-\bar{z}_3} \geq  e^{-\bar{c}_1}   + e^{-\bar{c}_2} + e^{-\bar{c}_3},\vspace{4pt}\\
      \bar{z}_1+\bar{z}_2+\bar{z}_3= \bar{c}_1+\bar{c}_2+\bar{c}_3=0.
    \end{array}
\end{equation}
Then
\begin{equation}\label{3,3}
    \bar{z}_1^2+\bar{z}_2^2+\bar{z}_3^2\geq \bar{c}_1^2+\bar{c}_2^2+\bar{c}_3^2\,.
\end{equation}
\end{formulation}\bigskip

Let us prove that Theorem \ref{exp} can be reformulated as Theorem \ref{exp2}. Indeed, let us assume that Theorem \ref{exp2} is valid and show that the statement of Theorem \ref{exp} also holds true. We denote by $s$ the sum $s= z_1+z_2+z_3= c_1+c_2+c_3$ and we designate
$$\bar z_i=z_i-\dfrac s3\,,\qquad \bar c_i=c_i-\dfrac s3\,\qquad (i=1,2,3).$$
Then, the real numbers $\bar z_i$ and $ \bar c_i$ satisfy the hypotheses of Theorem \ref{exp2} and we obtain
$ \bar{z}_1^2+\bar{z}_2^2+\bar{z}_3^2\geq \bar{c}_1^2+\bar{c}_2^2+\bar{c}_3^2\,$. This inequality is equivalent to
$$\sum_{i=1}^3 \Big(z_i-\dfrac s3\Big)^2 \geq \sum_{i=1}^3 \Big(c_i-\dfrac s3\Big)^2,$$
which, by virtue of the condition \eqref{3}$_3\,$, reduces to
$$  z_1^2+z_2^2+z_3^2\geq c_1^2+c_2^2+c_3^2\,.$$
Thus, Theorem \ref{exp} is also valid.

\bigskip

By virtue of the logical equivalence
$$(A \wedge B\,\,\,\Rightarrow\,\,\,C)\qquad\Leftrightarrow\qquad (\neg C \,\,\,\Rightarrow\,\,\, \neg A \vee \neg B)$$
for any statements $A,B,C$, we can formulate the inequality \eqref{3,3} (i.e., Theorem \ref{exp2}) in the following equivalent manner:

\begin{formulation}\label{forproof}
Let the real numbers $c_1, c_2, c_3 $ and $z_1, z_2, z_3 $ be such that
\begin{equation}\label{21}
    z_1+z_2+z_3= c_1+c_2+c_3=0\qquad\text{and}\qquad  z_1^2+z_2^2+z_3^2 < c_1^2+c_2^2+c_3^2\,.
\end{equation}
Then one of the following inequalities holds:
\begin{equation}\label{22}
    \begin{array}{l}
      e^{z_1}  +e^{z_2} +e^{z_3} <  e^{c_1}  +e^{c_2} +e^{c_3} \qquad\text{or}\vspace{6pt}\\
      e^{-z_1}   + e^{-z_2} + e^{-z_3} <  e^{-c_1}   + e^{-c_2} + e^{-c_3}.
    \end{array}
\end{equation}
\end{formulation}
We use the statement of Theorem \ref{forproof} for the proof.\\

Before continuing let us show that our new inequality is not a consequence of majorization and Karamata's inequality \cite{Karamata32}.
Consider $z=(z_1,\ldots,z_n)\in\R_+^n$ and $c=(c_1, \ldots, c_n)\in\R_+^n$ arranged already in decreasing order $z_1\ge z_2\ge \ldots \ge z_n$ and $c_1\ge c_2\ge \ldots \ge c_n$.
If
\begin{align}
  \label{majorization}
\sum_{i=1}^k z_{i}\ge \sum_{i=1}^k c_{i}\, ,\quad (1\le k\le n-1)\, ,\quad \sum_{i=1}^n z_{i}= \sum_{i=1}^n c_{i}\, ,
\end{align}
we say that $z$ majorizes $c$, denoted by $z\succ c$. The following result is well known \cite{Karamata32,Khan2012}\cite[p.89]{Hardy34}.
If $f:\R\mapsto\R$ is convex, then
\begin{align}
\label{karamata}
    z\succ c \quad \Rightarrow \quad \sum_{i=1}^n f(z_i)\ge \sum_{i=1}^n f(c_i)\, .
\end{align}
A function $g:\R^n\mapsto \R$ which satisfies
\begin{align}
\label{schur_convex}
    z\succ c \quad \Rightarrow \quad  g(z_1,\ldots,z_n)\ge  g(c_1,\ldots,c_n)
\end{align}
 is called Schur-convex. In Theorem \ref{exp} the convex function to be considered would be $f(t)=t^2$. Do conditions \eqref{3} (upon rearrangement of $z,c\in\R_+^3$ if necessary) yield already majorization $z\succ c$? This is not the case, as we explain now. Let the real numbers $z_1\geq z_2\geq z_3$ and $c_1\geq c_2\geq c_3 $ be such that
\begin{equation}\label{1}
    \begin{array}{l}
      e^{z_1}  +e^{z_2} +e^{z_3} \geq   e^{c_1}  +e^{c_2} +e^{c_3} ,\\
      e^{-z_1}   + e^{-z_2} + e^{-z_3} \geq  e^{-c_1}   + e^{-c_2} + e^{-c_3},\\
      z_1+z_2+z_3= c_1+c_2+c_3\, .
    \end{array}
\end{equation}
These conditions do not imply the majorization $z\succ c$,
\begin{equation}\label{2}
    z_1\geq c_1\,,\qquad z_1+z_2 \geq c_1+c_2\,,\qquad z_1+z_2+z_3= c_1+c_2+c_3\, .
    \end{equation}
Therefore, our inequality (i.e. $\, z_1^2+z_2^2+z_3^2\geq c_1^2+c_2^2+c_3^2\,\,$) does not follow from majorization in disguise.

\bigskip

Indeed, let $$z_1=\dfrac{1}{2}+\dfrac{0.95}{2\sqrt{3}}\,,\qquad z_2=\dfrac{1}{2}+\dfrac{0.85}{2\sqrt{3}}\,,\qquad z_3=-1-\dfrac{0.9}{\sqrt{3}}\,.$$
and
$$c_1=\dfrac{1}{2}+\dfrac{1}{2\sqrt{3}}\,,\qquad c_2=-\dfrac{1}{2}+\dfrac{1}{2\sqrt{3}}\,,\qquad c_3=-\dfrac{1}{\sqrt{3}}\,,$$
Then, we have $z_1> z_2> z_3  $ and $c_1> c_2> c_3 $, together with
\begin{equation*}
    \begin{array}{l}
      e^{z_1}  +e^{z_2} +e^{z_3} = 4.49497... >  3.57137... =   e^{c_1}  +e^{c_2} +e^{c_3} ,\\
      e^{-z_1}   + e^{-z_2} + e^{-z_3}= 5.50607... >  3.47107...=  e^{-c_1}   + e^{-c_2} + e^{-c_3},\\
      z_1+z_2+z_3= c_1+c_2+c_3=0,
    \end{array}
\end{equation*}
but the majorization inequalities \eqref{2} are not satisfied, since $z_1\,<\,c_1\,$.

\section{Proof of the inequality}
Of course we may assume without loss of generality that $c_1\geq c_2\geq c_3 $ and $z_1\geq z_2\geq z_3$ (and the same for $a_i,d_i,x_i,y_i$).

The proof begins with the crucial

\begin{lemma}\label{lemma1}
Let the real numbers $a\geq b\geq c  $ and $x\geq y\geq z  $ be such that
\begin{equation}\label{4}
    \begin{array}{l}
     a+b+c=x+y+z=0,\qquad a^2+b^2+c^2=x^2+y^2+z^2.
    \end{array}
\end{equation}
Then, the inequality
\begin{equation}\label{5}
    \begin{array}{l}
     e^a+e^b+e^c\leq e^x+e^y+e^z
    \end{array}
\end{equation}
is satisfied  if and only if the relation
\begin{equation}\label{6}
    \begin{array}{l}
     a\leq x
    \end{array}
\end{equation}
holds, or equivalently, if and only if
\begin{equation}\label{7}
    \begin{array}{l}
     c\leq z
    \end{array}
\end{equation}
holds.
\end{lemma}

\begin{proof}
Let us denote by $r:=\sqrt{\frac{2}{3}\,(a^2+b^2+c^2)}\,>0$. Then, from \eqref{4} it follows
\begin{equation*}
    \begin{array}{c}
      b+c=-a,\qquad b^2+c^2=\frac{3}{2}\,r^2-a^2, \vspace{6pt}\\
      y+z=-x,\qquad y^2+z^2=\frac{3}{2}\,r^2-x^2,
    \end{array}
\end{equation*}
and we find
\begin{equation}\label{8}
    \begin{array}{l}
      b=\frac{1}{2}\,\big(-a+\sqrt{3(r^2-a^2)}\,\big)\,,\qquad c=\frac{1}{2}\,\big(-a-\sqrt{3(r^2-a^2)}\,\big)\,, \vspace{6pt}\\
      y=\frac{1}{2}\,\big(-x+\sqrt{3(r^2-x^2)}\,\big)\,,\qquad z=\frac{1}{2}\,\big(-x-\sqrt{3(r^2-x^2)}\,\big)\,.
    \end{array}
\end{equation}
In view of \eqref{4} and $a\geq b\geq c  $ , $x\geq y\geq z  $ , one can show that
\begin{equation}\label{9}
\begin{array}{l}
     a,x\in\big[\,\dfrac{r}{2}\,,r\big],\qquad b,y\in\big[-\,\dfrac{r}{2}\,,\,\dfrac{r}{2}\,\big],\qquad c,z\in\big[-r\,,-\,\dfrac{r}{2}\,\big].
    \end{array}
\end{equation}
Indeed, let us verify the relations \eqref{9}. We have
\begin{equation*}
    \begin{array}{l}
       \dfrac{r}{2}\,\leq\,a\,\leq\,r \quad \Leftrightarrow\quad  \dfrac{1}{6}\,(a^2+b^2+c^2)\leq a^2\leq \dfrac{2}{3}\,(a^2+b^2+c^2)  \vspace{6pt}\\ \Leftrightarrow\quad  b^2+c^2\leq 5a^2\,\,\,\text{and}\,\,\, a^2\leq 2(b^2+c^2)
       \quad \Leftrightarrow\quad    b^2+(a+b)^2\leq 5a^2    \,\,\,\text{and}\,\,\, (b+c)^2\leq 2(b^2+c^2)  \vspace{6pt}
       \\ \Leftrightarrow\quad    4a^2-2ab-2b^2\geq 0     \,\,\,\text{and}\,\,\, b^2+c^2\geq 2bc
       \quad \Leftrightarrow\quad       2(a-b)(2a+b)\geq 0    \,\,\,\text{and}\,\,\, (b-c)^2\geq 0,
     \end{array}
\end{equation*}
which hold true since $a\geq b$ and $2a+b\geq a+b+c=0$. Similarly, we have
\begin{equation*}
    \begin{array}{l}
       -\dfrac{r}{2}\,\leq\,b\,\leq\,\dfrac{r}{2} \quad \Leftrightarrow\quad  b^2\leq \dfrac{r^2}{4} \quad \Leftrightarrow\quad
      4b^2\leq \dfrac{2}{3}\,(a^2+b^2+c^2)  \quad \Leftrightarrow\quad     5b^2\leq a^2+c^2   \vspace{6pt}\\ \Leftrightarrow\quad
       5b^2\leq a^2+(a+b)^2    \quad  \Leftrightarrow\quad    2a^2+2ab-4b^2\geq 0
       \quad \Leftrightarrow\quad       2(a-b)(a+2b)\geq 0  ,
     \end{array}
\end{equation*}
which holds true since $a\geq b$ and $a+2b\geq a+b+c=0$. Also, we have
\begin{equation*}
    \begin{array}{l}
       -r\,\leq\,c\,\leq\,-\dfrac{r}{2} \quad \Leftrightarrow\quad  r^2\geq c^2\geq \,\dfrac{r^2}{4}  \quad \Leftrightarrow\quad \dfrac{2}{3}\,(a^2+b^2+c^2) \geq c^2\geq \dfrac{1}{6}\,(a^2+b^2+c^2)  \vspace{6pt}\\ \Leftrightarrow\quad  2(a^2+b^2)\geq c^2\,\,\,\text{and}\,\,\, 5c^2\geq a^2+b^2
       \quad \Leftrightarrow\quad    2(a^2+b^2)\geq (a+b)^2    \,\,\,\text{and}\,\,\, 5(a+b)^2\geq a^2+b^2  \vspace{6pt}
       \\ \Leftrightarrow\quad     (a-b)^2\geq 0   \,\,\,\text{and}\,\,\, 4a^2+10ab+4b^2\geq 0
       \quad \Leftrightarrow\quad      (a-b)^2\geq 0    \,\,\,\text{and}\,\,\, 2(a+2b)(2a+b)\geq 0 ,
     \end{array}
\end{equation*}
which hold true since $a+2b\geq a+b+c=0$ and $2a+b\geq a+b+c=0$. One can show in the same way that $x\in\big[\,\dfrac{r}{2}\,,r\big]$ , $y\in\big[-\,\dfrac{r}{2}\,,\,\dfrac{r}{2}\,\big]$ , $z\in\big[-r\,,-\,\dfrac{r}{2}\,\big]$ , so that \eqref{9} has been verified.
\medskip

We prove now that the inequality \eqref{6} holds if and only if \eqref{7} holds. Indeed, using \eqref{8}$_{2,4}$ and \eqref{9} we get
\begin{equation*}
    \begin{array}{l}
       c\leq z\qquad \Leftrightarrow\qquad  -a-\sqrt{3(r^2-a^2)}\,\,\leq\,-x-\sqrt{3(r^2-x^2)} \qquad \Leftrightarrow\qquad\vspace{10pt}\\
       \qquad \Leftrightarrow\qquad \dfrac{a}{r}\,+\sqrt{3\Big(1-\big(\dfrac{a}{r}\big)^2\Big)} \geq \dfrac{x}{r}\,+\sqrt{3\Big(1-\big(\dfrac{x}{r}\big)^2\Big)} \qquad \Leftrightarrow\qquad a\leq x\,,
     \end{array}
\end{equation*}
since the function $\,t\,\mapsto\, t+\sqrt{3(1-t^2)}\,\,$ is decreasing for $t\in\big[\,\,\dfrac{1}{2}\,,\,1\big]$.\medskip

Let us prove next that the inequalities \eqref{5} and \eqref{6} are equivalent. To accomplish this, we introduce the function $\,\,f:\big[\,\dfrac{r}{2}\,,r\big]\rightarrow\mathbb{R}$ by
\begin{equation}\label{10}
    f(x)=e^x+ e^{\big(-x+\sqrt{3(r^2-x^2)}\,\big)/2} +e^{\big(-x-\sqrt{3(r^2-x^2)}\,\big)/2}.
\end{equation}
Taking into account \eqref{8} and \eqref{9}$_1$, the inequality \eqref{5} can be written equivalently as
\begin{equation}\label{11}
    f(a)\leq f(x)\,,
\end{equation}
which is equivalent to
\begin{equation*}
    a\leq x\,,
\end{equation*}
since the function $\,f\,$ defined by \eqref{10} is monotone increasing  on $\,\big[\,\dfrac{r}{2}\,,r\big]$, as we show next. To this aim, we denote by
$$\cos \varphi:=\dfrac{x}{r}\,\in\big[\,\,\dfrac{1}{2}\,,\,1\big],\qquad\text{i.e.}\qquad \varphi:=\arccos\Big(\dfrac{x}{r}\,\Big)\in\big[\,0,\,\dfrac{\pi}{3}\,\big].$$
Then, the function \eqref{10} can be written as
\begin{equation}\label{12}
    \begin{array}{l}
     f(x)=h(r,\varphi),\qquad \text{where}\qquad h:(0,\infty)\times \big[\,0,\,\dfrac{\pi}{3}\,\big]\rightarrow\mathbb{R}, \vspace{10pt} \\
      h(r,\varphi)=e^{r\cos\varphi} +e^{r\cos(\varphi+2\pi/3)} +e^{r\cos(\varphi-2\pi/3)}.
    \end{array}
\end{equation}
We have to show that $h(r,\varphi)$ is decreasing with respect to $\varphi\in\big[\,0,\,\dfrac{\pi}{3}\,\big]$. We compute the first derivative
\begin{equation}\label{13}
 \begin{array}{l}
    \dfrac{\partial}{\partial \varphi}\,h(r,\varphi) = -r\Big[ e^{r\cos\varphi} \sin\varphi +e^{r\cos(\varphi+2\pi/3)} \sin(\varphi+\frac{2\pi}{3}\,)+e^{r\cos(\varphi-2\pi/3)}\sin(\varphi-\frac{2\pi}{3}\,)\Big].
     \end{array}
\end{equation}
The function \eqref{13} has the same sign as the function
\begin{equation}\label{14}
    F(r,\varphi):=\dfrac{1}{r}\,e^{-r\cos\varphi}\, \dfrac{\partial}{\partial \varphi}\,h(r,\varphi),
\end{equation}
i.e. the function $F:(0,\infty)\times \big[\,0,\,\dfrac{\pi}{3}\,\big]\rightarrow\mathbb{R}$ given by
\begin{equation}\label{15}
\begin{array}{l}
    F(r,\varphi)= -\sin\varphi- e^{-r\sqrt{3}\sin(\varphi+\pi/3)} \sin(\varphi+\frac{2\pi}{3}\,)- e^{r\sqrt{3}\sin(\varphi-\pi/3)} \sin(\varphi-\frac{2\pi}{3}\,).
    \end{array}
\end{equation}
In order to show that $\,\,\,F(r,\varphi)\leq0\,\,\,$ for all $\,\,(r,\varphi)\in(0,\infty)\times\big[\,0,\,\dfrac{\pi}{3}\,\big]$,  we remark that
$\,\,\displaystyle{\lim_{r\searrow0}}F(r,\varphi)=0\,\,$ for fixed $\,\,\varphi\in\big[\,0,\,\dfrac{\pi}{3}\,\big]\,\,$
and we compute
\begin{equation*}
    \begin{array}{l}
       \dfrac{\partial}{\partial r}\,F(r,\varphi) = \sqrt{3}\Big[ e^{-r\sqrt{3}\sin(\varphi+\pi/3)} \sin(\varphi+\frac{\pi}{3}) \sin(\varphi+\frac{2\pi}{3})- e^{r\sqrt{3}\sin(\varphi-\pi/3)} \sin(\varphi-\frac{\pi}{3})\sin(\varphi-\frac{2\pi}{3})\Big]\vspace{10pt}\\
        \quad=
        \sqrt{3}\Big[   e^{-r\sqrt{3}\sin(\varphi+\pi/3)}\dfrac{1}{2}\big(-\cos(2\varphi+\pi)+ \cos\frac{\pi}{3}\,\big)
        - e^{r\sqrt{3}\sin(\varphi-\pi/3)}\dfrac{1}{2}\big(-\cos(2\varphi-\pi)+ \cos\frac{-\pi}{3}\,\big) \Big]
       \vspace{10pt}\\
      \qquad= \dfrac{\sqrt{3}}{2}\,\big(\cos2\varphi+\dfrac{1}{2}\,\big)\Big[ e^{-r\sqrt{3}\sin(\varphi+\pi/3)} - e^{r\sqrt{3}\sin(\varphi-\pi/3)}  \Big]\,\leq 0,
    \end{array}
\end{equation*}
since $\,\,\varphi\in\big[\,0,\,\dfrac{\pi}{3}\,\big]\,\,$ implies $\,\,\cos2\varphi\geq -\frac{1}{2}\,\,\,\,$ and $\,\,\,-\sin(\varphi+\frac{\pi}{3}\,)\leq \sin(\varphi-\frac{\pi}{3}\,)$.

Consequently, the function $\,F(r,\varphi)\,$ is decreasing with respect to $\,r\,$ and for any $\,\,(r,\varphi)\in(0,\infty)\times\big[\,0,\,\dfrac{\pi}{3}\,\big]$ we have that
\begin{equation}\label{16}
    F(r,\varphi)\,\,\leq \,\,\lim_{r\searrow0}F(r,\varphi)\,\,=\,0\,.
\end{equation}
From \eqref{14} and \eqref{16} it follows that $\,h(r,\varphi)\,$  is decreasing with respect to $\,\,\varphi\in\big[\,0,\,\dfrac{\pi}{3}\,\big]\,\,$. This means that $\,f(x)\,$ is increasing as a function of $\,x\in\big[\,\dfrac{r}{2}\,,r\big]$, i.e. the relation \eqref{11} is indeed equivalent to $\,a\leq x\,$ and the proof is complete.
\end{proof}

\begin{consequence}\label{cons2}
Let the real numbers $a\geq b\geq c$ and $x\geq y\geq z  $ be such that
\begin{equation*}
    \begin{array}{l}
     a+b+c=x+y+z=0,\qquad a^2+b^2+c^2=x^2+y^2+z^2.
    \end{array}
\end{equation*}
Then, one of the following inequalities holds:
\begin{equation}\label{17}
    \begin{array}{l}
     e^a+e^b+e^c\leq e^x+e^y+e^z\,,
    \end{array}
\end{equation}
or
\begin{equation}\label{18}
    \begin{array}{l}
     e^{-a}+e^{-b}+e^{-c}\leq e^{-x}+e^{-y}+e^{-z}\,.
    \end{array}
\end{equation}
The inequalities \eqref{17} and \eqref{18} are satisfied simultaneously if and only if $a=x$ , $b=y$ and $c=z$.
\end{consequence}

\begin{proof}
According to Lemma \ref{lemma1}, the inequality \eqref{17} is equivalent to
\begin{equation}\label{19}
    a\leq x\,,
\end{equation}
while the inequality \eqref{18} is equivalent to
\begin{equation}\label{20}
    -a\leq -x\,.
\end{equation}
Since one of the relations  \eqref{19} and \eqref{20}  must hold, we have proved that one of the inequalities \eqref{17} and \eqref{18} is satisfied. They are simultaneously satisfied if and only if both \eqref{19} and \eqref{20} hold true, i.e. $a=x$ (and consequently $b=y$, $c=z$).
\end{proof}

\begin{consequence}\label{cons3}
Let the real numbers $a\geq b\geq c$ and $x\geq y\geq z  $ be such that
\begin{equation*}
    \begin{array}{l}
     a+b+c=x+y+z=0,\qquad a^2+b^2+c^2=x^2+y^2+z^2\vspace{6pt}\\
     \quad\text{and}\qquad e^a+e^b+e^c = e^x+e^y+e^z\,.
    \end{array}
\end{equation*}
Then, we have $a=x$ , $b=y$ and $c=z$.
\end{consequence}

\begin{proof}
Since by hypothesis $e^a+e^b+e^c\leq e^x+e^y+e^z\,$ holds, we can apply the Lemma \ref{lemma1} to deduce $a\leq x$ and $c\leq z$.

On the other hand, by virtue of the inverse inequality $\,e^x+e^y+e^z\leq e^a+e^b+e^c\,$ and Lemma \ref{lemma1} we obtain $x\leq a$ and $z\leq c$. In conclusion, we get $a=x$ , $c=z$ and $b=y$.
\end{proof}

\begin{proof}[Proof of Theorem \ref{forproof}]
In order to prove \eqref{22} we define the real numbers
\begin{equation}\label{23}
    t_i=k\,z_i\quad (i=1,2,3)\qquad\text{where}\quad k=\sqrt{\dfrac{c_1^2+c_2^2+c_3^2}{z_1^2+z_2^2+z_3^2}}\,\,>1.
\end{equation}
Then we have
\begin{equation}\label{24}
    t_1+t_2+t_3= c_1+c_2+c_3=0\qquad\text{and}\qquad  t_1^2+t_2^2+t_3^2 = c_1^2+c_2^2+c_3^2\,.
\end{equation}
If we apply the Consequence \ref{cons2} for the numbers  $c_1\geq c_2\geq c_3 $ and $t_1\geq t_2\geq t_3 $ , then we obtain that
\begin{equation}\label{25}
   \begin{array}{l}
      e^{t_1}  +e^{t_2} +e^{t_3} \leq  e^{c_1}  +e^{c_2} +e^{c_3} \qquad\text{or}\vspace{6pt}\\
      e^{-t_1}   + e^{-t_2} + e^{-t_3} \leq  e^{-c_1}   + e^{-c_2} + e^{-c_3}.
    \end{array}
\end{equation}
In what follows, let us show that
\begin{equation}\label{26}
    e^{z_1}  +e^{z_2} +e^{z_3} < e^{t_1}  +e^{t_2} +e^{t_3}\,.
\end{equation}
Using the notations
$\rho:=\sqrt{\frac{2}{3}\,(z_1^2+z_2^2+z_3^2)}\,\,$ and
$$\cos\zeta:=\dfrac{z_1}{\rho}\,\in
\big[\,\dfrac{1}{2},\,1\,\big],\qquad\text{i.e.}\qquad \zeta:=\arccos \Big(\dfrac{z_1}{\rho}\Big)\in\big[\,0,\,\dfrac{\pi}{3}\,\big],$$
we have $k\rho:=\sqrt{\frac{2}{3}\,(t_1^2+t_2^2+t_3^2)}\,\,\,\,$ and $\,\,\,\cos\zeta=\dfrac{t_1}{k\rho}\,$.
With the help of
the function $\,h\,$ defined in \eqref{12}, we can write the inequality \eqref{26} in the form
$$e^{\rho\cos\zeta} +e^{\rho\cos(\zeta+2\pi/3)} +e^{\rho\cos(\zeta-2\pi/3)} < e^{k\rho\cos\zeta} +e^{k\rho\cos(\zeta+2\pi/3)} +e^{k\rho\cos(\zeta-2\pi/3)}\,,\qquad\text{or}$$
\begin{equation}\label{27}
    h(\rho,\zeta)< h(k\rho,\zeta),\qquad \forall \,(\rho,\zeta)\in (0,\infty)\times\big[\,0,\,\dfrac{\pi}{3}\,\big],\,\,\,k>1.
\end{equation}
The relation \eqref{27} asserts that the function $\,h\,$ defined in \eqref{12} is increasing with respect to the first variable $r\in(0,\infty)$. To show this, we compute the derivative
\begin{equation}\label{28}
\begin{array}{l}
    \dfrac{\partial}{\partial r}\,h(r,\varphi) = e^{r\cos\varphi} \cos\varphi +e^{r\cos(\varphi+2\pi/3)} \cos(\varphi+\frac{2\pi}{3}\,)+e^{r\cos(\varphi-2\pi/3)}\cos(\varphi-\frac{2\pi}{3}\,).
    \end{array}
\end{equation}
By virtue of the Chebyshev's sum inequality we deduce from \eqref{28} that
\begin{equation}\label{29}
    \dfrac{\partial}{\partial r}\,h(r,\varphi)\,>0.
\end{equation}
Indeed, the Chebyshev's sum inequality \cite[2.17]{Hardy34} asserts that: if $a_1\geq a_2\geq ...\geq a_n$ and $b_1\geq b_2\geq ...\geq b_n$ then
$$n\sum_{k=1}^na_kb_k\,\geq\, \Big(\sum_{k=1}^na_k\Big)\Big(\sum_{k=1}^nb_k\Big).$$
In our case, we derive the following result: for any real numbers $x,y,z$ such that $x+y+z=0$, the inequality
\begin{equation}\label{30}
    xe^x+ye^y+ze^z\geq \frac{1}{3}\,(x+y+z)(e^x+e^y+e^z)= 0\,,
\end{equation}
holds true, with equality if and only if $x=y=z=0$.

Applying the result \eqref{30} to the function \eqref{28} we deduce the relation \eqref{29}. This means that $\,h(r,\varphi)\,$   is an increasing function of $r$, i.e. the inequality \eqref{27} holds, and hence, we have proved \eqref{26}.

One can show analogously that the inequality
\begin{equation}\label{31}
    e^{-z_1}  +e^{-z_2} +e^{-z_3} < e^{-t_1}  +e^{-t_2} +e^{-t_3}\,
\end{equation}
is also valid. From \eqref{25}, \eqref{26} and \eqref{31} it follows that the assertion \eqref{22} holds true. Thus, the proof of Theorem \ref{forproof} is complete.
\end{proof}

Since the statements of the Theorems \ref{exp} and \ref{forproof} are equivalent, we have proved also the inequality \eqref{3bis}.

\begin{remark}\label{remark4}
The inequality \eqref{3bis} becomes an equality if and only if $z_i=c_i$ , $i=1,2,3$.
\end{remark}

\begin{proof}
Indeed, assume that   $z_1^2+z_2^2+z_3^2= c_1^2+c_2^2+c_3^2\,$. Then, we can apply the Consequence \ref{cons2} and we deduce that
\begin{equation}\label{32}
    \begin{array}{l}
      e^{z_1}  +e^{z_2} +e^{z_3} \leq   e^{c_1}  +e^{c_2} +e^{c_3} \,\,\quad\text{or}\quad\,\,
      e^{-z_1}   + e^{-z_2} + e^{-z_3} \leq  e^{-c_1}   + e^{-c_2} + e^{-c_3}.
    \end{array}
\end{equation}
Taking into account \eqref{3}$_{1,2}$ in conjunction with \eqref{32} we find
\begin{equation}\label{33}
    \begin{array}{l}
      e^{z_1}  +e^{z_2} +e^{z_3} =   e^{c_1}  +e^{c_2} +e^{c_3} \,\,\quad\text{or}\quad\,\,
      e^{-z_1}   + e^{-z_2} + e^{-z_3} =  e^{-c_1}   + e^{-c_2} + e^{-c_3}.
    \end{array}
\end{equation}
By virtue of \eqref{33} we can apply the Consequence \ref{cons3} to derive  $z_1= c_1$ and consequently $z_2= c_2$ , $z_3= c_3\,$.
\end{proof}

Let us prove the following version of the inequality \eqref{1bis} for two pairs of numbers $d_1,d_2$ and $x_1,x_2$ :
\begin{remark}\label{remark6}
If the real numbers $d_1\geq d_2>0$ and $x_1\geq x_2>0$ are such that
\begin{equation}\label{40}
    \begin{array}{l}
      x_1^2 +x_2^2 \geq  d_1^2 +d_2^2  \qquad\text{and}\qquad
      x_1x_2 \,=\, d_1d_2 \,=1\,,
    \end{array}
\end{equation}
then the inequality
\begin{equation}\label{41}
    (\log x_1)^2 + (\log x_2)^2  \geq (\log d_1)^2 + (\log d_2)^2
\end{equation}
holds true.
Note that the additional condition
\[
 \frac{1}{x_1^2}+\frac{1}{x_2^2}\geq \frac{1}{d_1^2}+\frac{1}{d_2^2}
\]
is automatically fulfilled.
\end{remark}

\begin{proof}
Since $\,x_1x_2 = d_1d_2 =1\,$ and $d_1\geq d_2>0$ , $x_1\geq x_2>0$,  we have $x_1\geq 1$ , $d_1\geq 1$ and
$$\log x_1=-\log x_2\geq 0\,,\qquad \log d_1=-\log d_2\geq 0\,,$$
so that the inequality \eqref{41} is equivalent to $\,\log x_1\geq \log d_1\,\,$, i.e. we have to show that $\,\,x_1\geq d_1\,$.

Indeed, if we insert $\,x_2=\dfrac{1}{x_1}\,\,$ and $\,d_2=\dfrac{1}{d_1}$ into the inequality \eqref{40}$_1$ then we find
$$x_1^2+\dfrac{1}{x_1^2}\,\geq \,d_1^2+\dfrac{1}{d_1^2}\,\,,$$
which means that $x_1\geq d_1$ since the function $\,\,t\mapsto t^2+\dfrac{1}{t^2}\,\,$ is increasing for $t\in [1,\infty)$. This completes the proof.
\end{proof}

\begin{proof}[Alternative proof of Remark \ref{remark6}]

Let $x_3=d_3=1$. Then \eqref{40} implies $x_1^2+x_2^2+x_3^2\geq d_1^2+d_2^2+d_3^2$ and $x_1x_2x_3=d_1d_2d_3=1$ as well as
\begin{equation}\label{42}
 x_1^2x_2^2+x_2^2x_3^2+x_1^2x_3^2=1+x_2^2+x_1^2\geq 1+d_2^2+d_1^2=d_1^2d_2^2+d_2^2d_3^2+d_1^2d_3^2,
\end{equation}
because $x_1^2x_2^2=1=d_1^2d_2^2$, and Theorem \ref{mitquadraten} provides the assertion.
\end{proof}

\section{Some counterexamples for weakened assumptions}

\begin{ex}
Unlike in the 2D case in Remark \ref{remark6}, for two triples of numbers the second condition \eqref{2}$_2\,$ of Theorem \ref{ohnequadrat}, namely $ y_1 y_2  +y_2 y_3 +y_1 y_3 \geq  a_1 a_2  +a_2 a_3 +a_1 a_3\,$, cannot be removed.
Let
\[y_1=e^6, y_2=1, y_3=e^{-6}, a_1=e^4, a_2=e^4, a_3=e^{-8}.\]
Then $y_1y_2y_3=a_1a_2a_3=1$ and
\[y_1 +y_2 +y_3 >e^6>e^2e^4>3e^4>a_1 +a_2 +a_3 ,\]
but
\[
 (\log y_1)^2+(\log y_2)^2+(\log y_3)^2=36+0+36<16+16+64=(\log a_1)^2+(\log a_2)^2+(\log a_3)^2.
\]
\end{ex}\bigskip

\begin{ex}
 The condition $y_1y_2y_3=a_1a_2a_3$ cannot be weakened to $y_1y_2y_3\geq a_1a_2a_3$.
Indeed, let $y_2=y_3=a_1=a_2=1$, $y_1=e$, $a_3= e^{-2}$. Then
\begin{equation*}
    \begin{array}{c}
       y_1+y_2+y_3=e+1+1\geq 1+1+e^{-2}=a_1+a_2+a_3\,,\\
       \quad y_1y_2+y_1y_3+y_2y_3=e+e+1\geq 1+e^{-2}+e^{-2}=a_1a_2+a_1a_3+a_2a_3\,,\\
       y_1y_2y_3=e\geq e^{-2}=a_1a_2a_3\,.
    \end{array}
\end{equation*}

But nevertheless
 \[
      (\log y_1)^2 + (\log y_2)^2 + (\log y_3)^2 = 1+0+0 < 0+0+4= (\log a_1)^2 + (\log a_2)^2 + (\log a_3)^2.
 \]

 A counterexample for the two variable case can be constructed analogously.
\end{ex}\bigskip

\begin{ex} Even with an analogous condition, the inequality \eqref{2bis} does not hold for $n=4$   numbers (without further assumptions). Indeed, let
\[y_1=e, y_2=y_3=e^7, y_4=e^{-15}, a_1=a_2=e^6, a_3=e^7, a_4=e^{-19}.\]
Then $y_1y_2y_3y_4=a_1a_2a_3a_4=1$.
Also
\[y_1+y_2+y_3+y_4= e+e^7+e^7+e^{-15}>0 + e^7+ 2e^6+ e^{-19}=a_1+a_2+a_3+a_4.\]
Furthermore
\[ y_1y_2+y_1y_3+y_1y_4+y_2y_3+y_2y_4+y_3y_4=e^8+e^8+e^{-14}+e^{14}+e^{-8}+e^{-8}\]
and
\[a_1a_2+a_1a_3+a_1a_4+a_2a_3+a_2a_4+a_3a_4= e^{12}+e^{13}+e^{-13}+e^{13}+e^{-13}+e^{-12}.\]
Since $e^2>2e+1$, we have $e^{14}>e^{13}+e^{13}+e^{12}$ and therefore
\[y_1y_2+y_1y_3+y_1y_4+y_2y_3+y_2y_4+y_3y_4 \geq a_1a_2+a_1a_3+a_1a_4+a_2a_3+a_2a_4+a_3a_4.\]
Nevertheless, for the sum of squared logarithms, the ``reverse'' inequality
\begin{align*}
 (\log y_1)^2+(\log y_2)^2+(\log y_3)^2+(\log y_4)^2=1+49+49+225 =324\\<482=36+36+49+361=(\log a_1)^2+(\log a_2)^2+(\log a_3)^2+(\log a_4)^2
\end{align*}
holds true.
\end{ex}\bigskip

\begin{ex} The inequality \eqref{2bis} does not remain true either, if the function $\log(y)$ is replaced by its linearization $(y-1)$.
Indeed, let $y_1=9,\,y_2=5,\,y_3=\frac{1}{45},\, a_1=10,\,a_2=1,\,a_3=\frac{1}{10}$. Then
\[y_1+y_2+y_3>14>11.1=a_1+a_2+a_3\] and
\[y_1y_2+y_1y_3+y_2y_3>45 \geq 11.1=a_1a_2+a_1a_3+a_2a_3.\]
But
\begin{align*}
(y_1-1)^2+(y_2-1)^2+(y_3-1)^2&=64+16+\left(\frac{44}{45}\right)^2\\
&<81 <9^2+0+\left(\frac{9}{10}\right)^2=(a_1-1)^2+(a_2-1)^2+(a_3-1)^2\, .
\end{align*}
\end{ex}\bigskip

\section{Conjecture for arbitrary $n$}
The structure of the inequality in dimensions $n=2$ and $n=3$ and extensive numerical sampling strongly suggest that the inequality holds for all $n\in\N$ if the $n$ corresponding conditions are satisfied, more precisely, in terms of the elementary symmetric polynomials
\begin{conj}
Let $n\in\N$ and $y_i,a_i>0$ for $i=1,\ldots, n$. If for all $i=1,\ldots,n-1$ we have $$e_i(y_1,\ldots, y_n)\geq e_i(a_1,\ldots,a_n) \text{ and } e_n(y_1,\ldots, y_n)=e_n(a_1,\ldots, a_n),$$ then $$\sum_{i=1}^n (\log y_i)^2\geq \sum_{i=1}^n (\log a_i)^2.$$
\end{conj}

\section{Applications}
The investigation in this paper has been motivated by some recent applications. The new sum of squared logarithm inequality is one of the fundamental tools in deducing a novel optimality result in matrix analysis and the conditions in the form \eqref{2neff} had been deduced in the course of that work. Optimality in the matrix problem suggested the sum of squared logarithm inequality. Indeed, based on the present result in \cite{Neff_Nagatsukasa_logpolar13} it has been shown that for all invertible $Z\in\C^{3\times 3}$ and for any definition of the matrix logarithm as possibly multivalued solution $X\in\C^{3\times 3}$ of $\exp X=Z$ it holds
\begin{align}
\label{log_optimal}
   \min_{Q^*Q=I}\norm{\log Q^* Z}_F^2&=\norm{\log U_p^* Z}_F^2= \norm{\log H}_F^2      \, ,\notag\\
   \min_{Q*Q=I}\norm{\sym \log Q^* Z}_F^2&=\norm{\sym \log U_p^* Z}_F^2=\norm{\log H}_F^2\, ,
\end{align}
where $\sym X=\frac{1}{2}(X+X^*)$ is the Hermitian part of $X\in\C^{3\times 3}$ and $U_p$ is the unitary factor in the polar decomposition of $Z$ into unitary and Hermitian positive definite matrix $H$
\begin{align}
     Z=U_p\, H\, .
\end{align}
This result \eqref{log_optimal} generalizes the fact that for any complex logarithm and for all $z\in\C\setminus\{0\}$
\begin{align}
 \min_{\vartheta\in(-\pi,\pi]}\abs{\log_{\C}[e^{-i\vartheta} z]}^2&=\abs{\log_\R\abs{z}}^2\, ,\quad
 \min_{\vartheta\in(-\pi,\pi]}\abs{\Re\log_{\C}[e^{-i\vartheta} z]}^2=\abs{\log_\R\abs{z}}^2\, .
\end{align}
The optimality result \eqref{log_optimal} can now also be viewed as another characterization of the unitary factor in the polar decomposition. In addition, in a forthcoming contribution \cite{Neff_Osterbrink_hencky13} we use \eqref{log_optimal} to calculate the geodesic distance of the isochoric part of the deformation gradient $ \frac{F}{\det{F}^{\frac{1}{3}}}\in\SL(3,\R))$ to $\SO(3,\R))$ in the canonical left-invariant Riemannian metric on $\SL(3,\R)$, to the effect that
\begin{align}
    \dist_{\rm geod}^2(\frac{F}{\det{F}^{\frac{1}{3}}},\SO(3,\R) )=\norm{\dev_3\log\sqrt{F^T F}}_F^2\, .
\end{align}
where $\dev_3 X=X-\frac{1}{3}\tr{X}\, I$ is the orthogonal projection of $X\in \R^{3\times 3}$ to trace free matrices. Thereby, we provide a rigorous geometric justification for the preferred use of the Hencky-strain measure $\norm{\log\sqrt{F^TF}}_F^2$ in nonlinear elasticity and plasticity theory \cite{Hencky28}.

\bibliographystyle{amsplain}

{\footnotesize
\bibliography{literatur1}
}
\end{document}